\documentclass{amsart}
\usepackage{amssymb,calrsfs,enumerate}
\usepackage{hyperref}
\newcommand{\rhs}{\mathbf{H}_{\mathbf{R}}}
\newcommand{\chs}{\mathbf{H}_{\mathbf{C}}}
\newcommand{\qhs}{\mathbf{H}_{\mathbf{Q}}}
\newcommand{\ohs}{\mathbf{H}_{\mathbf{O}}}
\theoremstyle{plain}
\newtheorem{theorem}{Theorem}
\newtheorem{lemma}[theorem]{Lemma}
\newtheorem{proposition}[theorem]{Proposition}

\theoremstyle{definition}
\newtheorem{definition}[theorem]{Definition}
\newtheorem{example}[theorem]{Example}
\theoremstyle{remark}
\newtheorem{remark}[theorem]{Remark}
\newtheorem*{acknowledgments}{Acknowledgments}
\begin{document}
\title[Minimal hypersurfaces in $\qhs^n$]{Complete minimal hypersurfaces\\in quaternionic hyperbolic space}
\author[J.L. Orjuela Ch.]{Jaime Leonardo Orjuela Chamorro}
\curraddr{Departamento de Matem\'atica\\Universidade Federal de S\~ao Carlos\\S\~ao Carlos, SP~13565-905\\Brazil}
\email{jlorjuelac@gmail.com}
\date{August 2013}
\thanks{The author was partially supported by FAPESP grant \mbox{2007/08513-6} and by CAPES~PNPD grant \mbox{02885/09-3}.}
\begin{abstract}
We construct new examples of embedded, complete minimal hypersurfaces in quaternionic hyperbolic space and also some minimal foliations. We introduce fans and construct analytic deformations of bisectors.
\end{abstract}
\maketitle
\section{Introduction}
It is a natural procedure to try to transfer results obtained in complex hyperbolic geometry to quaternionic hyperbolic geometry. On the other hand, the richness of the algebraic and geometric structure of rank one symmetric spaces makes these Riemannian manifolds reasonable candidates to test various geometric problems. In the literature there are two prominent examples of minimal hypersurfaces in complex hyperbolic space $\chs^n$, namely, \emph{bisectors} and \emph{fans}. Bisectors are introduced by Giraud~\cite{giraud} and then by Mostow~\cite{mostow1} (he calls a bisector a \emph{spinal surface}). In \cite{goldman0} Goldman makes a systematic study of this class of hypersurfaces. Fans are introduced by Goldman and Parker in \cite{goldman-parker1}. Bisectors and fans are used to construct fundamental polyhedra for discrete groups of isometries of $\chs^n$; see, for instance, \cite{mostow1}, \cite{goldman-parker0}, \cite{goldman-parker1}, \cite{gusevskii-parker0} and \cite{gusevskii-parker1}. Bisectors are defined for any hyperbolic space (in general for metric spaces), simply as the geometric loci of all points equidistant from two distinct given points. They are introduced as replacements for totally geodesic real hypersurfaces in non-real hyperbolic spaces and are used to construct Dirichlet fundamental polyhedra (see \cite{apanasov-kim1}). We note that bisectors in quaternionic hyperbolic space $\qhs^n$ are ruled minimal hypersurfaces of cohomogeneity one, all congruent and diffeomorphic to $\mathbf{R}^{4n-1}$ (see $\S$\ref{sec-bisectors}). We introduce fans in $\qhs^n$ following \cite{goldman-parker1}. Fans can be viewed as limit cases of bisectors (see Example~\ref{fans-example}). We notice that fans in $\qhs^n$ are homogeneous, ruled, minimal hypersurfaces all congruent and diffeomorphic to $\mathbf{R}^{4n-1}$ (see $\S$\ref{sec-fans}).

In this paper, we construct new examples of complete minimal hypersurfaces in $\qhs^n$ using the method of equivariant differential geometry introduced by Hsiang-Lawson in \cite{hsiang-lawson1}.

An isometric action of a Lie group on a Riemannian manifold is called \emph{polar} if there exists a connected, complete (necessarily totally geodesic) submanifold intersecting each orbit orthogonally; such a submanifold is called a \emph{section}. In an analogy with \cite{gorodski-gusevskii1}, where $\chs^n$ is studied, we consider several subgroups of the isometry group of $\qhs^n$ which are adapted to its Iwasawa decomposition. These subgroups define polar actions of cohomogeneity two on $\qhs^n$. As sections, we always have a totally geodesic real hyperbolic plane. We then compute the canonical projection  (i.e., the orbital invariants)  and write the reduced ordinary differential equation (\ref{reduced-ode}) in the orbit space (which is embedded naturally and isometrically in the section), whose solutions are the curves generating minimal hypersurfaces in $\qhs^n$. Our main results are the following:
\begin{theorem}\label{theorem-one}
\begin{enumerate}[(i)]
\item\label{theorem-one-i} For each $m=1,\dots,n-1$, let $H=Sp(m)\times Sp(n-m)\times\{1\}$ be embedded diagonally into the isometry group $\mathbf{P}Sp(n,1)$ of $\qhs^n$. Then there exist infinitely many non-congruent embedded, complete, minimal hypersurfaces in $\qhs^n$ that are $H$-equivariant (and hence, of cohomogeneity one).
\item\label{theorem-one-ii} For each $m=2,\dots,n-1$, let $H=Sp(n-m)\times\{1\}\times Sp(m-1,1)$ be embedded diagonally into the isometry group $\mathbf{P}Sp(n,1)$ of $\qhs^n$. Then there exist infinitely many non-congruent embedded, complete, minimal hypersurfaces in $\qhs^n$ that are $H$-equivariant (and hence, of cohomogeneity one).
\end{enumerate}
\end{theorem}
It is not hard to show that the hypersurfaces constructed in Theorem~\ref{theorem-one}, case~(\ref{theorem-one-i}), are of the diffeomorphic type of $\mathbf{R}^{4m}\times\mathbf{S}^{4n-4m-1}$, with a homogeneous \emph{ideal boundary}  of the diffeomorphic type of $\mathbf{S}^{4m-1}\times\mathbf{S}^{4n-4m-1}$ (product of $\mathbf{Q}$-spheres). Here the ideal boundary of an embedded submanifold $M$ of $\qhs^n$ is defined to be $\partial M:=\bar{M}\cap\partial\qhs^n$, where $\bar{M}$ denotes the closure of $M$ relative to $\qhs^n\cup\partial\qhs^n$.

We show that bisectors are non-rigid as minimal hypersurfaces:
\begin{theorem}\label{theorem-two}
Bisectors in $\qhs^n$ admit non-trivial deformations preserving minimality. Namely, each bisector belongs to an analytic one-parameter family of minimal hypersurfaces such that no other member in the family is a bisector.
\end{theorem}
We also construct some interesting minimal foliations of $\qhs^n$:
\begin{theorem}\label{theorem-three}
\begin{enumerate}[(i)]
\item\label{theorem-three-i} For each $m = 1,\dots,n-1$ there exists a foliation of $\qhs^n$ by minimal hypersurfaces diffeomorphic to $\mathbf{R}^{4n-1}$, invariant by a one-parameter group of transvections, and such that each leaf has an ideal boundary of the homeomorphic type of a pinched Hopf manifold of type $(4m-1,4n-4m-1)$.
\item\label{theorem-three-ii} There exists a foliation of $\qhs^n$ by minimal hypersurfaces diffeomorphic to $\mathbf{R}^{4n-1}$, invariant by a one-parameter group of transvections, and such that each leaf has an ideal boundary of the homeomorphic type of a bouquet of two spheres $\mathbf{S}^{4n-2}$.
\item\label{theorem-three-iii} There exists a foliation of $\qhs^n$ by homogeneous, ruled, minimal hypersurfaces diffeomorphic to $\mathbf{R}^{4n-1}$, invariant by a group of parabolic isometries, and such that each leaf has an ideal boundary of the homeomorphic type of $\mathbf{S}^{4n-2}$. Namely, each leaf is a fan.
\end{enumerate}
\end{theorem}
Here the \emph{pinched Hopf manifold} of type $(k,l)$, for $k$, $l$ positive integers, is defined to be the topological space obtained by contracting a fiber of the trivial bundle $\mathbf{S}^k\times\mathbf{S}^l\to\mathbf{S}^l$ to a point. For instance, for $k=l=1$, we have a pinched torus.

The foliations from Theorem~\ref{theorem-three} induce non-smooth foliations of the ideal boundary $\partial\qhs^n\approx\mathbf{S}^{4n-1}$, pinched at the point at infinity, as it follows from the construction in the proof. Notice that a congruence between two foliations of $\qhs^n$ as in Theorem~\ref{theorem-three} induces a homeomorphism between the respective boundary foliations. Therefore, we see that the foliations in case~(\ref{theorem-three-i}) for $m = 1,\dots,n-1$ together with the foliations in case~(\ref{theorem-three-ii}) and case~(\ref{theorem-three-iii}) are pairwise non-congruent, since the type of the boundaries of their leaves is different. Also it is interesting to recall that, as it follows from the proof of Theorem~\ref{theorem-three} and Proposition~\ref{fans-foliation}, each leaf of the foliation in case~(\ref{theorem-three-iii}) of Theorem~\ref{theorem-three} is isometric to the homogeneous minimal hypersurface $S(0,V_0)$ constructed in \cite{berndt1}, and, in fact, the foliation is built selecting the unique minimal leaf of $\mathfrak{F}(\theta,V_0)$ (which is a specific translate of $S(\theta,V_0)$) for each $\theta\in(-\frac{\pi}{2},\frac{\pi}{2})$.

We want to indicate that the techniques in this paper can be extended to investigate minimal hypersurfaces in the \emph{octonionic hyperbolic plane} $\ohs^2=F_4^*/Spin(9)$. In fact, Chen~\cite{chen1} classified connected Lie subgroups of $F_4^*$. Also, Kollross~\cite{kollross} classified connected \emph{reductive algebraic} subgroups of $F_4^*$ whose actions on $\ohs^2$ are polar. In this way we can construct several families of subgroups of $F_4^*$, adapted to its Iwasawa decomposition, acting polarly on $\ohs^2$ with cohomogeneity two. For an analysis of the real and complex hyperbolic spaces see \cite{docarmo-dajczer1} and \cite{gorodski-gusevskii1}, respectively.
\section{Quaternionic hyperbolic space and its isometry group}
\subsection{Models of $\qhs^n$}\label{models}
Let $\mathbf{Q}$ be the non-commutative normed division algebra of the quaternions. If $q$ is a quaternion we write $q=q_0+iq_1+jq_2+kq_3$ and $\bar{q}=q_0-iq_1-jq_2-kq_3$, where $q_0$, $q_1$, $q_2$, $q_3\in\mathbf{R}$ and $\{1,i,j,k\}$ is the canonical orthonormal basis of $\mathbf{Q}$.

Consider the (right) $\mathbf{Q}$-module $\mathbf{Q}^{n+1}$ of all column vectors $X$ with coefficients $X_1$, $X_2,\dots,X_{n+1}\in\mathbf{Q}$, equipped with the indefinite Hermitean form
\begin{equation}\label{hermitean-form}
\langle X,Y \rangle = \bar{X}_1Y_1+\cdots+\bar{X}_nY_n-\bar{X}_{n+1}Y_{n+1},
\end{equation}
for all $X$, $Y\in\mathbf{Q}^{n+1}$. It determines the following regions of $\mathbf{Q}^{n+1}$
\begin{align*}
V_{+} =& \{X\in\mathbf{Q}^{n+1}: \langle X,X \rangle > 0\}\\
V_{0} =& \{X\in\mathbf{Q}^{n+1}: \langle X,X \rangle = 0\}\\
V_{-} =& \{X\in\mathbf{Q}^{n+1}: \langle X,X \rangle < 0\}.
\end{align*}
The projectivization $\mathbf{P}V_{-}$ is the \emph{quaternionic hyperbolic space} $\qhs^n$. On $\qhs^n$ we define the Riemannian metric
\begin{equation*}\label{metric}
ds^2 = -\frac{4}{\langle X,X \rangle^2}\det\begin{bmatrix}\langle X,X \rangle&\langle dX,X \rangle\\ \langle X,dX \rangle&\langle dX,dX \rangle\end{bmatrix}.
\end{equation*}
So that sectional curvature of $\qhs^n$ lies between $-1$ and $-\frac{1}{4}$. Also the distance function $d$ is given by
\[d(\mathbf{P}X,\mathbf{P}Y) = 2~\mathrm{arccosh}\frac{|\langle X,Y \rangle|}{\sqrt{\langle X,X \rangle\langle Y,Y \rangle}}.\]
The projectivization $\mathbf{P}V_{0}$ defines the \emph{ideal boundary} of $\qhs^n$ which we denote by $\partial\qhs^n$. The point at infinity $\infty\in\partial\qhs^n$ is given by the vector $X^{\infty}\in\mathbf{Q}^{n+1}$ with coordinates $X_l^{\infty}=0$, for $l=1,\dots, n-1$ and $X_n^{\infty}=X_{n+1}^{\infty}=1$.

The group of transformations of $\mathbf{Q}^{n+1}$ that preserve the form (\ref{hermitean-form}) is the non-compact Lie group
\begin{equation*}\label{isometry-group}
Sp(n,1)=\{A\in\mathbf{GL}(n+1,\mathbf{Q}):A^*I_{n,1}A=I_{n,1}\},
\end{equation*}
where $A^*$ denotes the conjugate transpose of the matrix $A$, $I_{n,1}=\begin{bmatrix}I_n&0\\0&-1\end{bmatrix}$ and  $I_n$ is the identity matrix of order $n$. It is clearly seen that $Sp(n,1)$ acts transitively by isometries on $\qhs^n$. This action is not effective because the center of $Sp(n,1)$ (real scalar matrices) acts trivially. Hence $\mathbf{P}Sp(n,1)=Sp(n,1)/\{\pm I_{n+1}\}$ is the isometry group of $\qhs^n$. We recall also that $Sp(n,1)$ acts naturally on $\partial\qhs^n$.

We have several models for the quaternionic hyperbolic space. On $\mathbf{Q}^n$ we consider the Hermitian definite form given by $(x,y)=\sum_{l=1}^n\bar{x}_ly_l$ for all $x,y\in\mathbf{Q}^n$ and write $|x|=\sqrt{(x,x)}$.
Note that the condition $\langle X,X \rangle < 0$ implies $X_{n+1}\not=0$. Hence the diffeomorphism $\mathbf{P}X\mapsto x$ given by
\[\mbox{$x_l=X_lX_{n+1}^{-1}$, for $l=1,\dots,n$}\]
identify $\qhs^n$ and $\partial\qhs^n$ with the (open) unit disc
$\mathbf{D}^n=\left\{x\in\mathbf{Q}^n:|x|<1\right\}$
and the unit sphere $\mathbf{S}^{4n-1}=\left\{x\in\mathbf{Q}^n:|x|=1\right\}$, respectively. We call $x_1,\dots,x_n$ \emph{affine coordinates} for $\qhs^n$. In affine coordinates we have that
\[ds^2=4~\frac{(1-|x|^2)|dx|^2+|(dx,x)|^2}{(1-|x|^2)^2}\]
and
\begin{equation}\label{distance-affine-coordinates}
d(x,y) = 2~\mathrm{arccosh}\frac{|1-(x,y)|}{\sqrt{(1-|x|^2)(1-|y|^2)}}.
\end{equation}
Quaternionic hyperbolic space can also be realized as an unbounded domain. In fact, for all $u\in\mathbf{Q}^n$ let $u'\in\mathbf{Q}^{n-1}$ the projection of  $u$ on the first $n-1$ coordinates. Then the \emph{Cayley transformation}
\begin{equation}\label{cayley-transformation}
\begin{cases}
\zeta'  &= x'(1-x_n)^{-1}\\
\zeta_n &= \mbox{\small$\frac{1}{2}$}(1+x_n)(1-x_n)^{-1}
\end{cases}
\end{equation}
gives a diffeomorphism between $\mathbf{D}^n$ and the \emph{Siegel domain} which is defined by
\[\mathcal{S}^n=\{\zeta\in\mathbf{Q}^n:|\zeta'|^2 - 2\Re(\zeta_n) < 0\}.\]
In the Siegel domain the
ideal boundary is $\partial\mathcal{S}^n=\{\zeta\in\mathbf{Q}^n:|\zeta'|^2 - 2\Re(\zeta_n) = 0\}\cup\{\infty\}$.

Given a geodesic $\gamma$ in $\qhs^n$ parametrized by arc length, the levels of the associated \emph{Busemann function} $h_{\gamma}(p):=\lim\limits_{s\to+\infty}\{d(p,\gamma(s))-s\}$ are called \emph{horospheres} and they foliate $\qhs^n$.
In the Siegel domain we consider the geodesic $\gamma$ with end points $0$ and $\infty$, namely, $\gamma(s)=\begin{bmatrix}0'\\\frac{1}{2}\mathrm{e}^{s}\end{bmatrix}$. It follows from an easy computation from (\ref{distance-affine-coordinates}) and (\ref{cayley-transformation}) that
\begin{equation*}\label{busemann-function}
h_{\gamma}(\zeta)=-\ln(2\Re(\zeta_n)-|\zeta'|^2).
\end{equation*}
Hence, the horospheres $H_{\alpha}=h_{\gamma}^{-1}(-\ln\alpha)$ are given by
\begin{equation*}\label{horoespheres}
\mbox{$H_{\alpha}=\{\zeta\in\mathcal{S}^n:2\Re(\zeta_n)-|\zeta'|^2=\alpha\}$, $\alpha>0$.}
\end{equation*}
So we obtain the \emph{horospherical coordinates} $(\omega,\alpha,\beta)\in\mathbf{Q}^{n-1}\times\mathbf{R}_{+}\times\Im(\mathbf{Q})$, where
\begin{equation*}\label{horospherical-coordinates}
\omega=\zeta' \text{ and } \alpha+\beta=2\zeta_n-|\zeta'|^2.
\end{equation*}
In horospherical coordinates the ideal boundary is $\left(\mathbf{Q}^{n-1}\times\{0\}\times\Im(\mathbf{Q})\right)\cup\{\infty\}$. We have that $H_{\alpha}=\mathbf{Q}^{n-1}\times\{\alpha\}\times\Im(\mathbf{Q})$, so each horosphere is diffeomorphic to $\partial\qhs^n-\{\infty\}$.
Finally
\begin{equation*}\label{metric-horospherical-coordinates}
ds^2 = \frac{d\alpha^2+|d\beta -2\Im(d\omega,\omega)|^2+4\alpha|d\omega|^2}{\alpha^2}.
\end{equation*}
\subsection{Iwasawa decomposition of $Sp(n,1)$}\label{iwasawa}
The Iwasawa decomposition for the non-compact Lie group $Sp(n,1)$ is $Sp(n,1)=\mathcal{H}^{4n-1}\cdot\mathbf{R}_{+}\cdot(Sp(n)\cdot Sp(1))$, where
\begin{align*}\label{iwasawa-decomposition}
Sp(n)\cdot Sp(1) &= \left\{\begin{bmatrix}B&0\\0&\lambda\end{bmatrix}\in Sp(n,1):\mbox{$B\in Sp(n)$ and $\lambda\in Sp(1)$}\right\}/\{\pm I_{n+1}\},\\
\mathbf{R}_{+} &= \left\{\psi_t=\begin{bmatrix}I_{n-1}&0&0\\0&\cosh t&\sinh t\\0&\sinh t&\cosh t\end{bmatrix}:t\in\mathbf{R}\right\},
\end{align*}
and
\begin{equation*}\label{heisenberg-translations}
\mathcal{H}^{4n-1} = \left\{h(\xi,\nu)=\begin{bmatrix}I_{n-1}&\xi&-\xi\\-\xi^*&1-\frac{1}{2}(|\xi|^2+\nu)&\frac{1}{2}(|\xi|^2+\nu)\\
-\xi^*&-\frac{1}{2}(|\xi|^2+\nu)&1+\frac{1}{2}(|\xi|^2+\nu)\end{bmatrix}:
\begin{matrix}\xi\in\mathbf{Q}^{n-1}\\\nu\in\Im(\mathbf{Q})\end{matrix}\right\}.
\end{equation*}
The group $\mathcal{H}^{4n-1}$ is called \emph{quaternionic Heisenberg group} and its elements viewed as isometries of $\qhs^n$ are called \emph{Heisenberg translations}. 

From the geometrical point of view, in the disc model, these subgroups can be described as follows. The group $Sp(n)\cdot Sp(1)$ is the isotropy subgroup at the origin (the base-point) $0\in\mathbf{D}^n$. Consider the geodesic $\gamma(s)=\begin{bmatrix}0'\\\tanh s\end{bmatrix}$. Its centralizer is the subgroup $Sp(n-1)\cdot Sp(1)$ identified with
\begin{equation*}\label{centralizer-geodesic}
\left\{\begin{bmatrix}B&0&0\\0&\lambda&0\\0&0&\lambda\end{bmatrix}\in Sp(n,1):\mbox{$B\in Sp(n-1)$ and $\lambda\in Sp(1)$}\right\}/\{\pm I_{n+1}\}.
\end{equation*}
The group $\mathbf{R}_{+}$ is the one-parameter group of transvections along $\gamma$, namely, each $\psi_t$ is a dilatation which maps $H_{\alpha}$ to $H_{\mathrm{e}^{2t}\alpha}$. The group $\mathcal{H}^{4n-1}$ acts simply and transitively on each horosphere, in particular on the ideal boundary fixing $\infty$, since $H_0$ is identified with $\partial\qhs^n-\{\infty\}$.

Algebraically speaking, we describe the action of some subgroups in horospherical coordinates:
\begin{align*}
B\cdot\lambda(\omega,\alpha,\beta)&=(B\omega\lambda^{-1},\alpha,\lambda\beta\lambda^{-1})\mbox{, for $B\cdot\lambda\in Sp(n-1)\cdot Sp(1)$},\\
\psi_t(\omega,\alpha,\beta)&=(\mathrm{e}^t\omega,\mathrm{e}^{2t}\alpha,\mathrm{e}^{2t}\beta)\mbox{, for $t\in\mathbf{R}$}\\
h(\xi,\nu)(\omega,\alpha,\beta)&=(\xi+\omega,\alpha,\nu+\beta+2\Im(\xi^*\omega))\mbox{, for $\xi\in\mathbf{Q}^{n-1}$ and $\nu\in\Im(\mathbf{Q}$)}.
\end{align*}
Finally, we recall that $\mathcal{H}^{4n-1}$ is isomorphic to $\mathbf{Q}^{n-1}\times\Im(\mathbf{Q})$ equipped with the product $(\xi_1,\nu_1)\cdot(\xi_2,\nu_2):=(\xi_1+\xi_2,\nu_1+\nu_2+2\Im(\xi_1^*\xi_2))$, and its Lie algebra is isomorphic to $\mathbf{Q}^{n-1}\oplus\Im(\mathbf{Q})$ where $[v_1+z_1,v_2+z_2]:=4\Im(v_1^*v_2)$.
\section{Bisectors and fans}
Let $E$ be a $\mathbf{Q}$-submodule of $\mathbf{Q}^{n+1}$ with dimension $m+1$. Suppose that $E$ intersects $V_-$. Then we have that $\mathbf{P}(E\cap V_-)$ is a totally geodesic submanifold of $\qhs^n$ of real dimension $4m$ called \emph{$\mathbf{Q}$-subspace}. In particular, a quaternionic $2$-plane determines a \emph{$\mathbf{Q}$-line} whose ideal boundary is called \emph{chain}. The ideal boundary of a \emph{$\mathbf{Q}$-hyperplane} is called \emph{hyperchain}. If $L$ is a $\mathbf{Q}$-hyperplane, then the \emph{inversion} at $L$ is the involutive isometry of $\qhs^n$ which has $L$ as its  set of fixed points.
Given a $\mathbf{Q}$-hyperplane $L$ there exists a vector $\lambda\in V_+$ such that the inversion at $L$ is induced by the transformation
\begin{equation*}\label{inversion-geral}
\mbox{$X\mapsto X-2\lambda\frac{\langle\lambda,X\rangle}{\langle\lambda,\lambda\rangle}$, for all $X\in\mathbf{Q}^{n+1}$.}
\end{equation*}
So $L=\mathbf{P}(\lambda^{\perp}\cap V_-)$, where $\lambda^{\perp}=\{X\in\mathbf{Q}^{n+1}:\langle\lambda,X\rangle=0\}$.
For instance, is not hard to check that, in the disc model, $L=\{x\in\mathbf{D}^n:x_n=0\}$ is a $\mathbf{Q}$-hyperplane whose associated inversion fixes $x'$ and maps $x_n$ to $-x_n$. Passing to horospherical coordinates, $L=\{(\omega,\alpha,\beta):\mbox{$|\omega|^2+\alpha=1$ and $\beta=0$}\}$ and its inversion is given by
\begin{equation}\label{inversion}
\iota(\omega,\alpha,\beta)=\left(\omega(\alpha + |\omega|^2 + \beta)^{-1},\frac{\alpha}{|\alpha+|\omega|^2+\beta|^2},\frac{-\beta}{|\alpha+|\omega|^2+\beta|^2}\right).
\end{equation}
\subsection{Bisectors}\label{sec-bisectors}
Following \cite{apanasov-kim1,goldman-parker1} we present the basic notions related to bisectors.

Let $p_1,p_2\in\qhs^n$ be two distinct points. Then the \text{bisector} $\mathcal{B}=\mathcal{B}(p_1,p_2)$ consist of all points in $\qhs^n$ equidistant from $p_1$ and $p_2$:
\begin{equation*}\label{bisector}
\mathcal{B}=\{p\in\qhs^n:d(p,p_1)=d(p,p_2)\}.
\end{equation*}
Let $\Sigma=\Sigma(p_1,p_2)$ be the unique $\mathbf{Q}$-line containing $p_1$ and $p_2$. We say that $\Sigma$ is the \emph{$\mathbf{Q}$-spine} of $\mathcal{B}$. The \emph{spine} or (\emph{real spine}) of $\mathcal{B}$ is defined by
$\sigma=\sigma(p_1,p_2)=\Sigma\cap\mathcal{B}$. Note that $\sigma\simeq\rhs^3$ is a (real hyperbolic) bisector in $\Sigma\simeq\rhs^4$ orthogonal to the geodesic containing $p_1$ and $p_2$.

Let $\Pi_{\Sigma}:\qhs^n\to\Sigma$ be the \emph{orthogonal projection} on $\Sigma$, i.e. for all $s\in\Sigma$ the preimage $\Pi_{\Sigma}^{-1}(s)$ is equal to the $\mathbf{Q}$-hyperplane orthogonal to $\Sigma$ at $s$. It is due to Mostow-Giraud that (see \cite[Theorem 5.1.1]{goldman0} or \cite[Theorem 2.1]{apanasov-kim1}):
\[\mathcal{B}=\Pi_{\Sigma}^{-1}(\sigma).\]
The $\mathbf{Q}$-hyperplanes $\Pi_{\Sigma}^{-1}(s)$, $s\in\sigma$ are called \emph{slices} of $\mathcal{B}$.
In particular, $\mathcal{B}$ is a (real) hypersurface ruled by $\mathbf{Q}$-hyperplanes. Bisectors in $\qhs^n$  are all congruent because $Sp(n,1)$ acts transitively on the set of all equidistant points in $\qhs^n$. We can prove that the slices (respectively, $\mathbf{Q}$-spine and spine) depend intrinsically on the hypersurface $\mathcal{B}$ and not on the pair $p_1$, $p_2$. Actually, a bisector is completely determined by its spine.
The ideal boundary $\partial\mathcal{B}$ is diffeomorphic to $\mathbf{S}^{4n-2}$ and is called \emph{spinal sphere}. The foliation of $\mathcal{B}$ by its slices induces a foliation of its spinal sphere by hyperchains. The ideal boundary $\partial\sigma$ is diffeomorphic to $\mathbf{S}^2$ and is called \emph{vortical sphere}.
\begin{example}\label{canonical-bisector}
Consider the points $p_1=\begin{bmatrix}0'\\ \frac{1}{2}k\end{bmatrix}$ and $p_2=\begin{bmatrix}0'\\-\frac{1}{2}k\end{bmatrix}$ in $\mathbf{D}^n$. From (\ref{distance-affine-coordinates}) we get that 
\begin{equation*}
\mbox{$\cosh\left(\frac{1}{2}d(x,p_1)\right)=\frac{|2-\bar{x}_nk|}{\sqrt{3(1-|x|^2)}}$\, and\, $\cosh\left(\frac{1}{2}d(x,p_2)\right)=\frac{|2+\bar{x}_nk|}{\sqrt{3(1-|x|^2)}}$,}
\end{equation*}
for all $x\in\mathbf{D}^n$. Thus, $\mathcal{B}=\mathcal{B}(p_1,p_2)=\{x\in\mathbf{D}^n:\Re(kx_n)=0\}$. Passing to horospherical coordinates we obtain $\mathcal{B}=\{(\omega,\alpha,\beta):\Re(k\beta)=0\}$. Thus, the $\mathbf{Q}$-spine and spine of $\mathcal{B}$ are given respectively by
$\Sigma=\{(\omega,\alpha,\beta):\omega=0\}$ and $\sigma=\{(\omega,\alpha,\beta):\mbox{$\omega=0$ and $\Re(k\beta)=0$}\}$.
The orthogonal projection is given by $\Pi_{\Sigma}(\omega,\alpha,\beta)=(0,\alpha+|\omega|^2,\beta)$, so the slices are of the form
\[\mbox{$S_{(0,a,b)}=\{(\omega,\alpha,\beta):\mbox{$\alpha+|\omega|^2=a$ and $\beta=b$}\}$, for all $(0,a,b)\in\sigma$}.\]
The stabilizer in $\mathbf{P}Sp(n,1)$ of $\mathcal{B}$ is equal to the stabilizer of $\sigma$, which is isomorphic to
\[(\mathbf{Z}_2\ltimes N\cdot\mathbf{R}_{+})\cdot Sp(n-1)\cdot T^1,\]
where
\[N=\{h(\xi,\nu)\in\mathcal{H}^{4n-1}:\mbox{$\xi=0$ and $\Re(k\nu)=0$}\},\]
\[Sp(n-1)\cdot T^1=\{(B,\lambda)\in Sp(n-1)\times Sp(1):\mbox{$\lambda=\mathrm{e}^{kt}$, $t\in\mathbf{R}$}\}/\{\pm I_{n+1}\}\]
and $\mathbf{Z}_2$ is the cyclic group generated by the inversion in a slice of $\mathcal{B}$. The component at identity is $N\cdot\mathbf{R}_{+}\cdot Sp(n-1)\cdot T^1$. It follows that $\mathcal{B}$ has cohomogeneity one. In fact, $N\cdot\mathbf{R}_+$ acts free and transitively on $\sigma$ and $Sp(n-1)$ acts on each slice by `rotations' pointwise fixing $\sigma$. (In \cite{apanasov-kim1} it is stated without proof that bisectors in real, complex, quaternionic and octonionic hyperbolic space have cohomogeneity $0$, $1$, $3$ and $7$, respectively. It seems to us that the authors have overlooked the $N$-factor.)
\end{example}
\begin{proposition}\label{bisector-minimal}
Bisectors are minimal hypersurfaces.
\end{proposition}
\begin{proof}
Let $\mathbb{H}$ be the mean curvature vector field of $\mathcal{B}$. Fix a slice $S$ of $\mathcal{B}$ and consider the inversion $\iota$ in $S$. Note that $\iota$ stabilizes $\mathcal{B}$, so $\mathbb{H}$ is $\iota$-invariant. Also, we have that $d\iota$ maps a normal vector at $S$ to its opposite. Then $\mathbb{H}$ is identically zero on $S$, since $S$ is the fixed point set of $\iota$. Finally $\mathcal{B}=\cup_{s\in S}G(s)$, where $G$ is the stabilizer of $\mathcal{B}$. So using again $g$-invariance of $\mathbb{H}$ for $g\in G$ we see that $\mathbb{H}\equiv 0$ on $\mathcal{B}$.

We note that an independent proof can be obtained from Remark~\ref{special-loxodromic-remark}~(\ref{special-loxodromic-explicit-solutions}) since bisectors are all congruent. 
\end{proof}
\subsection{Fans}\label{sec-fans}
For complex hyperbolic space, fans are introduced in \cite{goldman-parker1}. Following some of the ideas of this work we introduce fans in $\qhs^n$. First, consider the \emph{pencil} of all $\mathbf{Q}$-lines in $\qhs^n$ which are asymptotic to $\infty$. The pencil has a natural structure of $(n-1)$-dimensional \emph{quaternionic affine space}. In fact, in horospherical coordinates the $\mathbf{Q}$-line containing $p_0=(\omega_0,0,\beta_0)$ and $\infty$ in its ideal boundary is given by
\begin{equation}\label{pencil}
\Sigma(p_0)=\{(\omega,\alpha,\beta):\omega=\omega_0\}.
\end{equation}
Note that actually $\Sigma(p_0)$ depends solely on $\omega_0\in\mathbf{Q}^{n-1}$. Now, consider the projection $\Pi:\qhs^n\to\mathbf{Q}^{n-1}$ given (in horospherical coordinates) by \[\Pi(\omega,\alpha,\beta)=\omega.\]
\begin{definition}\label{fan-definition}
If $F\subset\mathbf{Q}^{n-1}$ is a real affine hyperplane, then its preimage
\[\mathcal{F}=\Pi^{-1}(F)\]
is called \emph{fan} with vertex at $\infty$. Since the inversion $\iota$ given by (\ref{inversion}) interchanges $(0,0,0)$ and $\infty$, we can use $\iota$ together with Heisenberg translations to define fans with vertex at an arbitrary point in $\partial\qhs^n$.
\end{definition}
\begin{proposition}\label{fans}
Fans are homogeneous, ruled, minimal hypersurfaces all congruent, diffeomorphic to $\mathbf{R}^{4n-1}$ and have ideal boundary homeomorphic to the sphere $\mathbf{S}^{4n-2}$.
\end{proposition}
\begin{proof}
By congruence it suffices to consider fans with vertex at $\infty$. Using Heisenberg translations and rotations of $Sp(n-1)$ we see that fans with vertex at $\infty$ are all congruent. It follows from Definition~\ref{fan-definition} that fans are diffeomorphic to $\mathbf{R}^{4n-1}$ with ideal boundary homeomorphic to $\mathbf{S}^{4n-2}$. Next, consider the fan
\begin{equation}\label{fan-a}
\mathcal{F}=\{(\omega,\alpha,\beta):\Re(\omega_{n-1})=0\}.
\end{equation}
We have that $\mathcal{F}=\bigcup_{\nu\in\Im(\mathbf{Q})}M_{\nu}$, where $M_{\nu}=\{(\omega,\alpha,\beta):\omega_{n-1}=\nu\}$. Therefore $\mathcal{F}$ is a (real) hypersurface ruled by the $\mathbf{Q}$-hyperplanes $M_{\nu}$, $\nu\in\Im(\mathbf{Q})$. On the other hand $\mathcal{F}$ is the $G$-orbit of the base-point $(0,1,0)$, where
\[G=\{h(\xi,\nu)\cdot\psi_t\in\mathcal{H}^{4n-1}\cdot\mathbf{R}_+:\Re(\xi_{n-1})=0\}.\]
Hence fans are homogeneous. Minimality follows as in proof of Proposition~\ref{bisector-minimal} replacing $S$ by $M_0$ and using homogeneity.

We note that an independent proof of the minimality can be obtained from Remark~\ref{special-parabolic-remark}~(\ref{special-parabolic-explicit-solutions}).
\end{proof}
Recall that $S:=\mathcal{H}^{4n-1}\cdot\mathbf{R}_+$ is a \emph{solvable} Lie Group equipped with the \emph{left-invariant metric} induced from the inner product on its Lie algebra $\mathfrak{s}\subset\mathfrak{sp}(n,1)$. From $\S$\ref{iwasawa} it follows that $\mathfrak{s}=\mathfrak{n}\oplus\mathfrak{a}$, where  $\mathfrak{a}=\mathbf{R}$ and $\mathfrak{n}$ is the Lie algebra of $\mathcal{H}^{4n-1}$ which decomposes as $\mathfrak{v}\oplus\mathfrak{z}$, where $\mathfrak{v}$ and $\mathfrak{z}$ are identified with $\mathbf{Q}^{n-1}$ and $\Im(\mathbf{Q})$, respectively. Of course, $S$ acts freely and transitively on $\qhs^n$. In the following paragraph and Proposition~\ref{fans-foliation}, we identify $S$ with $\qhs^n$ via the \emph{orbit map} through $(\omega,\alpha,\beta)=(0,1,0)$.

Proposition 1 in \cite{berndt1} says that codimension-one Lie subalgebras of $\mathfrak{s}$ are of the form $\mathfrak{s}(\theta,V_0):=(\cos\theta V_0+\sin\theta A)^{\perp}$, for some $\theta\in\left(-\frac{\pi}{2},\frac{\pi}{2}\right]$ and some unit vector $V_0\in\mathfrak{v}$, where $A$ is a non-zero vector in $\mathfrak{a}$. Let $S(\theta,V_0)$ be the Lie subgroup of $S$ with Lie algebra $\mathfrak{s}(\theta,V_0)$. We have that $S(\theta,V_0)$ acts isometrically on $S$ with cohomogeneity one and its orbits form a \emph{Riemannian foliation} on $S$ which is denoted by $\mathfrak{F}(\theta,V_0)$. Clearly, $S\left(\frac{\pi}{2},V_0\right)=\mathcal{H}^{4n-1}$ and each leaf of $\mathfrak{F}\left(\frac{\pi}{2},V_0\right)$ is a horosphere. Now, let $\gamma:\mathbf{R}\to S$ be the geodesic with $\gamma(0)$ being the identity and $\gamma'(0)=V_0$. \cite[$\S$4]{berndt1} shows that the action of $S(0,V_0)$ on $S$ is polar and $\gamma(\mathbf{R})$ is a section. Moreover $S(\theta,V_0)=\gamma(t)^{-1}S(0,V_0)\gamma(t)$, where $\theta=-\arcsin\tanh\frac{t}{2}$. In particular, for each $\theta\in\left(-\frac{\pi}{2},\frac{\pi}{2}\right)$ the Riemannian foliation $\mathfrak{F}(\theta,V_0)$ is the left translation of $\mathfrak{F}(0,V_0)$ by $\gamma(t)$, where $t=2\mathrm{arctanh}(-\sin\theta)$. 
\begin{proposition}\label{fans-foliation}
Consider the notation above. If  $\mathcal{F}$ is the fan given by (\ref{fan-a}), then there exists some $V_0\in\mathfrak{v}$ such that $\mathcal{F}$ is isometric to $S(0,V_0)$.
\end{proposition}
\begin{proof}
In fact, consider the vector $v_0=\begin{bmatrix}0\\1\end{bmatrix}$ in $\mathbf{Q}^{n-1}$. Thus, the vector field
\[V_0=\left.\frac{d}{dt}\right|_{t=0}h(tv_0,0)\]
belongs in $\mathfrak{v}$. Note that $S(0,V_0)$ is equal to the group $G$ in the proof of Proposition~\ref{fans}. Hence, $\mathcal{F}$ is the $S(0,V_0)$-orbit through $(0,1,0)$.
\end{proof}
We conclude this section showing that fans can be seen as limits of bisectors as their vortical spheres collapse to the vertex of the fan:
\begin{example}\label{fans-example}
Consider the bisector $\mathcal{B}=\{(\omega,\alpha,\beta):\Re(k\beta)=0\}$ in Example~\ref{canonical-bisector}. For all $t\in\mathbf{R}$ let $h_t:=h(tv_0,0)$ as in the proof of Proposition~\ref{fans-foliation}. Applying the one-parameter group of Heisenberg translations $(h_t)_{t\in\mathbf{R}}$ to $\mathcal{B}$, we obtain the one-parameter family of bisectors $(\mathcal{B}_t)_{t\in\mathbf{R}}$ which are given by
\[\mathcal{B}_t=\{(\omega,\alpha,\beta):\Re(k(\beta-2t\omega_{n-1}))=0\}.\]
Their respective vortical spheres are
\[\partial\sigma_t=\{(\omega,\alpha,\beta):\mbox{$\omega=tv_0$ and $\Re(k\beta)=\alpha=0$}\}.\]
Then, letting $t\to+\infty$ we obtain the fan
\begin{equation*}\label{fan-b}
\mathcal{F}=\{(\omega,\alpha,\beta):\Re(k\omega_{n-1})=0\}.
\end{equation*}
Finally, consider the inversion $\iota$ given by (\ref{inversion}). Then the fan $\mathcal{F}'=\iota(\mathcal{F})$ has vertex at $(0,0,0)$ and
\begin{equation*}\label{fan-c}
\mathcal{F}'=\{(\omega,\alpha,\beta):\mbox{$\alpha>0$ and $\Re(k\omega_{n-1}(\alpha+|\omega^2|+\beta)^{-1})=0$}\}.
\end{equation*}
Writing $\beta=i\beta_1+j\beta_2+k\beta_3$ and $\omega_l=\omega_{l,0}+i\omega_{l,1}+j\omega_{l,2}+k\omega_{l,3}$, with $\beta_r,\omega_{l,m}\in\mathbf{R}$, for $r=1$, $2$, $3$, $m=0$, $1$, $2$, $3$  and $l=1,\dots,n-1$, we have that $(\omega,\alpha,\beta)\in\mathcal{F}'$ if and only if $\alpha>0$ and
\begin{align*}\label{fan-singularity}
0 =& \omega_{n-1,3}\left(\alpha+\sum_{l=1}^{n-1}(\omega_{l,0}^2+\omega_{l,1}^2+\omega_{l,2}^2+\omega_{l,3}^2)\right)\\
 +&\omega_{n-1,2}\beta_1-\omega_{n-1,1}\beta_2-\omega_{n-1,0}\beta_3.
\end{align*}
By computing the partial derivatives it is not hard to show that the only singular point of the above equation for $\alpha\geqq 0$ is $(\omega,\alpha,\beta)=(0,0,0)$ (compare with \cite[p.536]{goldman-parker1}).
\end{example}
\section{The basic reduction}
\subsection{Preliminaries}
Let $H$ be a Lie group acting properly and isometrically on a Riemannian manifold $M$. The principal orbit type theorem in \cite[p.86-87]{palais-terng} asserts that the union $M_r$ of all principal orbits in $M$ is an open, dense, invariant submanifold with connected orbit space $\Delta_r=H\backslash M_r$. The orbital metric on $\Delta_r$ is the one that makes $M_r\to\Delta_r$ into a Riemannian submersion. In fact, since the decomposition of $M$ into orbit types is locally finite, one gets a stratified Riemannian submersion $M\to\Delta=H\backslash M$. Moreover, when $(H,M)$ is polar then $\Delta$ is isometric to the \emph{orbifold} $W\backslash\Sigma$, where $\Sigma$ is a section and $W=W(\Sigma)$ is its \emph{generalized Weyl group}. In this case the orbital metric is the induced metric of $M$ in $\Sigma$.

The \emph{volume functional} measures the volume element of the principal orbits. It is a continuous function on $\Delta$, differentiable on $\Delta_r$ and null on the singular orbits (see \cite{hsiang-lawson1}). The volume functional is computed as follows. Let $H/K$ be the principal orbit type. Then each principal orbit is a homogeneous Riemannian manifold of $H/K$-type with induced metric from $M$ which is completely determined by its value at the base-point. Let $\mathfrak{h}$ be the Lie algebra of $H$ equipped with the corresponding $\mathrm{Ad}_K$-invariant inner product, and let $\mathfrak{p}$ be the orthogonal complement in $\mathfrak{h}$ to the Lie algebra of $K$. Choose an orthogonal basis $\{X_1,\dots,X_m\}$ for $\mathfrak{p}$ and denote by $X_i^*$, $i=1,\dots,m$, the induced Killing vector fields on $M$. The volume functional $\mathcal{V}$ of $(H,M)$ is given by (see \cite{hsiang})
\begin{lemma}\label{volume}
The volume element of the orbit through $p\in M$ is given by $\mathcal{V}(p)d(H/K)$, where
\begin{equation*}\label{volume-functional}
\mathcal{V}(p)=|X_1^*(p)\wedge\cdots\wedge X_m^*(p)|
\end{equation*}
and $d(H/K)$ is the volume element of $H/K$.
\end{lemma}

Assume in the following that the action of $H$ on $M$ have \emph{cohomogeneity} two, i.e., the principal orbits have codimension two. Let $\Gamma$ be a $H$-equivariant hypersurface and let $\gamma=H\backslash\Gamma$ its generating curve in $\Delta$. Then (see \cite{back-docarmo-hsiang1,hsiang})
\begin{lemma}[Reduced ODE]\label{reduction-theorem}
The mean curvature $\mathbf{h}$ of $\Gamma$ is given by the following formula
\begin{equation}\label{reduced-ode}
\mathbf{h}= \kappa_g - \frac{d}{d\xi}\ln(\mathcal{V}),
\end{equation}
where $\kappa_g$ is  the geodesic curvature and $\xi$ is the positive unit normal of $\gamma$ respect to the orbital metric.
\end{lemma}

Further, suppose that $(H,M)$ is polar. The orbifold $\Delta$ has, in general, non-empty boundary $\partial\Delta$ which is composed by strata with codimension one and two (corresponding to singular orbits). The reduced ODE~(\ref{reduced-ode}) is singular in $\partial\Delta$, since the volume functional is identically null on this set. However we can consider solutions emanating orthogonally from the codimension one strata. In fact, these solutions are the most interesting.
\begin{lemma}\label{boundary}
Let $z_0\in\partial\Delta$ be a point in a codimension one stratum. Then there exists a unique solution $\gamma_{z_0}$ of (\ref{reduced-ode}) with initial condition $\gamma_{z_0}(0)=z_0$ and it is necessarily perpendicular to $\partial\Delta$ at $z_0$. Furthermore, there exists a neighborhood of $(z_0,0)$ in $\partial\Delta\times\mathbf{R}$ such that $\gamma_z(t)=\gamma(z,t)$ is analytic. Finally, the generated hypersurface is smooth.
\end{lemma}
Lemma \ref{boundary} is obtained from the following technical result, which is proved by the well know technique of power series substitution and majoration \cite[Proposition 1]{hsiang-hsiang2}
\begin{lemma}\label{analytic}
There exists a unique analytic solution $y=y(t,x)$ for the following system (\ref{power-series}) which is a convergent power series of $(t,x)$ in a neighborhood of $(0,0)$ and $y(t,0)=0$, $\frac{dy}{dx}(t,0)=p(t,0)=0$
\begin{equation}\label{power-series}
\begin{aligned}
\frac{dy}{dx}=&p\\
x\frac{dp}{dx}=&\lambda p + a_{0,1,0,0} x +\psi(t,x,y,p), 
\end{aligned}
\end{equation}
where $\lambda$ is not a positive integer, $t$ is a parameter, and 
\[\psi=\sum_{\overset{l+m+n+\nu\geq2}{m+n+\nu\geq1}}a_{l,m,n,\nu}t^lx^my^np^{\nu}.\] 
\end{lemma}
\subsection{The elliptic case}\label{elliptic}
For each $m=1,\dots,n-1$, we shall consider the following subgroup of $Sp(n,1)$:
\[H=Sp(m)\times Sp(n-m)=\left\{\begin{bmatrix}A&0&0\\0&B&0\\0&0&1\end{bmatrix}\in Sp(n,1):\begin{matrix}A\in Sp(m)\\B\in Sp(n-m)\end{matrix}\right\}.\]
Using the disk model, it is easy to see that
\[\Sigma=\{x\in\mathbf{D}^n:\mbox{$x_m$, $x_n\in\mathbf{R}$ and $x_l=0$ for $l\neq m$, $n$}\}\simeq\rhs^2\]
is a section for $(H,\qhs^n)$ and the orbit space is isometric to
\[\Delta=\{(u,v)\in\mathbf{R}^2:\mbox{$u^2+v^2\leq 1$ and $u$, $v\geq 0$}\},\]
where the canonical projection $\mathbf{D}^n\to\Delta$ is given by
\[x\mapsto\left(\sqrt{\sum_{l=1}^{m}|x_l|^2},\sqrt{\sum_{l=m+1}^n|x_l|^2}\right).\]
The orbital metric on $\Delta$ is
\[ds^2=\frac{4}{1-u^2-v^2}\{(1-v^2)du^2+2uvdudv+(1-u^2)dv^2\},\]
that is, the (real) hyperbolic metric of constant curvature $-\frac{1}{4}$. Also, from Lemma~\ref{volume}, the volume functional at $(u,v)\in\Delta$ is
\[\mathcal{V}=\frac{u^{4m-1}v^{4n-4m-1}}{(1-u^2-v^2)^{\frac{4n+1}{2}}}.\]
Passing to polar coordinates $u=\tanh r\cos\theta$ and $v=\tanh r\sin\theta$, we can write
\begin{align*}
ds^2 &= 4(dr^2+\sinh^2rd\theta^2)\\
\mathcal{V} &= (\sinh r)^{4n-5}(\sinh2r)^3(\sin\theta)^{4n-8m}(\sin2\theta)^{4m-1},
\end{align*}
for $r\geq 0$ and $0\leq\theta\leq\frac{\pi}{2}$. Finally, Lemma~\ref{reduction-theorem} gives
\begin{proposition}\label{elliptic-case}
Let $\gamma(s)=(r(s),\theta(s))$ be a curve in $\Delta$ parametrized by arc length, and let $\sigma(s)$ be the angle between $\frac{\partial}{\partial r}$ and the tangent unit vector $\frac{d\gamma}{ds}$. Then the hypersurface $\Gamma$ with $H\backslash\Gamma=\gamma$ has mean curvature $\mathbf{h}$ if and only if
\begin{equation}\label{elliptic-ode}
\begin{aligned}
\frac{dr}{ds} &= \frac{1}{2}\cos\sigma\\
\frac{d\theta}{ds} &= \frac{1}{2}\frac{\sin\sigma}{\sinh r}\\
\frac{d\sigma}{ds} &= \left((2n-4m)\cot\theta+(4m-1)\cot 2\theta\right)\frac{\cos\sigma}{\sinh r}\\
                   &-((2n-2)\coth r+3\coth 2r)\sin\sigma +\mathbf{h}.
\end{aligned}
\end{equation}
\end{proposition}
\begin{remark}\label{elliptic-remark}
\begin{enumerate}[(i)]
\item\label{elliptic-ode-boundary} System~(\ref{elliptic-ode}) is singular on the boundary $\{(r,\theta):\mbox{$\theta=0$ or $\theta=\frac{\pi}{2}$}\}$. However, there exist solutions of (\ref{elliptic-ode}) emanating orthogonally from points in there (see Lemma~\ref{boundary}).
\item\label{elliptic-explicit-solutions}(\emph{Explicit solutions}) The curve $\theta\equiv\arctan\sqrt{\frac{4n-4m-1}{4m-1}}$ is a solution of (\ref{elliptic-ode}) with $\mathbf{h}\equiv0$. The generated hypersurface is a cone over $\mathbf{S}^{4n-4m-1}\times\mathbf{S}^{4m-1}$, therefore a singular hypersurface. Also, we have that for $a>0$ the curves $r\equiv a$ are solutions of system (\ref{elliptic-ode}) with $\mathbf{h}\equiv\pm((2n-2)\coth a+3\coth2a)$ and they generate metric spheres centered at the base-point.
\end{enumerate}
\end{remark}
\subsection{The loxodromic case}\label{loxodromic}
We next consider the following subgroup of $Sp(n,1)$ for each $m=2,\dots,n-1$:
\[H=Sp(n-m)\times Sp(m-1,1)=\left\{\begin{bmatrix}A&0&0\\0&1&0\\0&0&B\end{bmatrix}\in Sp(n,1):\begin{matrix}A\in Sp(n-m)\\B\in Sp(m-1,1)\end{matrix}\right\}.\]
Using the disk model, it is not difficult to see that
\[\Sigma=\{x\in\mathbf{D}^n:\mbox{$x_{n-m}$, $x_{n-m+1}\in\mathbf{R}$ and $x_l=0$ if $l\neq n-m$, $n-m+1$}\}\simeq\rhs^2\]
is a section for $(H,\qhs^n)$ and the orbit space is isometric to
\[\Delta=\{(u,v)\in\mathbf{R}^2:\mbox{$u^2+v^2\leq 1$ and $u$, $v\geq 0$}\},\]
where the canonical projection $\mathbf{D}^n\to\Delta$ is given by
\[x\mapsto\left(\frac{|x_{n-m+1}|}{\sqrt{1-\sum\limits_{l=n-m+2}^{n}|x_l|^2}},
\frac{\sqrt{\sum\limits_{l=1}^{n-m}|x_l|^2}}{\sqrt{1-\sum\limits_{l=n-m+2}^{n}|x_l|^2}}\right).\]
The orbital metric on $\Delta$ is
\[ds^2=\frac{4}{1-u^2-v^2}\{(1-v^2)du^2+2uvdudv+(1-u^2)dv^2\},\]
that is, the (real) hyperbolic metric of constant curvature $-\frac{1}{4}$. Also, from Lemma~\ref{volume}, the volume functional at $(u,v)\in\Delta$ is
\[\mathcal{V}=\frac{u^3v^{4n-4m-1}}{(1-u^2-v^2)^{\frac{4n+1}{2}}}.\]
Passing to polar coordinates $u=\tanh r\cos\theta$ and $v=\tanh r\sin\theta$, we can write
\begin{align*}
ds^2 &= 4(dr^2+\sinh^2rd\theta^2)\\
\mathcal{V} &= (\sinh r)^{4n-8m+3}(\sinh2r)^{4m-1}(\sin\theta)^{4n-4m-4}(\sin2\theta)^3,
\end{align*}
for $r\geq 0$ and $0\leq\theta\leq\frac{\pi}{2}$. Finally, Lemma~\ref{reduction-theorem} gives
\begin{proposition}\label{loxodromic-case}
Let $\gamma(s)=(r(s),\theta(s))$ be a curve in $\Delta$ parametrized by arc length, and let $\sigma(s)$ be the angle between $\frac{\partial}{\partial r}$ and the tangent unit vector $\frac{d\gamma}{ds}$. Then the hypersurface $\Gamma$ with $H\backslash\Gamma=\gamma$ has mean curvature $\mathbf{h}$ if and only if
\begin{equation}\label{loxodromic-ode}
\begin{aligned}
\frac{dr}{ds} &= \frac{1}{2}\cos\sigma\\
\frac{d\theta}{ds} &= \frac{1}{2}\frac{\sin\sigma}{\sinh r}\\
\frac{d\sigma}{ds} &= \left((2n-2m-2)\cot\theta+3\cot 2\theta\right)\frac{\cos\sigma}{\sinh r}\\
                   &-((2n-4m+2)\coth r+(4m-1)\coth 2r)\sin\sigma +\mathbf{h}.
\end{aligned}
\end{equation}
\end{proposition}
\begin{remark}\label{loxodromic-remark}
\begin{enumerate}[(i)]
\item\label{loxodromic-ode-boundary} System~(\ref{loxodromic-ode}) is singular on the boundary $\{(r,\theta):\mbox{$\theta=0$ or $\theta=\frac{\pi}{2}$}\}$. However, there exist solutions of (\ref{loxodromic-ode}) emanating orthogonally from points in there (see Lemma~\ref{boundary}).
\item\label{loxodromic-explicit-solutions} (\emph{Explicit solutions})
The curve $\theta\equiv\arctan\sqrt{\frac{4n-4m-1}{3}}$ is a solution of (\ref{loxodromic-ode}) with $\mathbf{h}\equiv0$. The generated hypersurface is the product of $\mathbf{R}^{4m-4}$ with a cone over $\mathbf{S}^{4n-4m-1}\times\mathbf{S}^3$, therefore a singular hypersurface. Also, for $a>0$ the curves $r\equiv a$ are solutions of (\ref{loxodromic-ode}) with $\mathbf{h}\equiv\pm((2n-4m+2)\coth a+(4m-1)\coth2a)$ and they generate tubes of constant radius around a $\mathbf{Q}$-subspace of real dimension $4m-4$ through the base-point.
\end{enumerate}
\end{remark}
\subsection{The special loxodromic case}\label{special-loxodromic}
In $\S$\ref{sec-bisectors} we described the stabilizer of a bisector. Now, consider its  subgroup:
\[H=N\cdot\mathbf{R}_+\cdot Sp(n-1).\]
Using the disk model, is not hard to see that
\[\Sigma=\{x\in\mathbf{D}^n:\mbox{$x_n$, $x_{n-1}\in k\mathbf{R}$ and $x_l=0$ if $l\neq n$, $n-1$}\}\simeq\rhs^2\]
is a section for $(H,\qhs^n)$ and the orbit space is isometric to
\[\Delta=\{(u,v)\in\mathbf{R}^2:\mbox{$u^2+v^2\leq 1$ and $v\geq 0$}\}.\]
The canonical projection $x\mapsto(u,v)$ is given by
\begin{equation*}\label{canonical-projection-bisector}
u =\frac{2x_{n,3}}{1-|x_n|^2+\sqrt{|1-x_n|^2|1+x_n|^2-4x_{n,1}^2-4x_{n,2}^2}}
\end{equation*}
and
\[v =\frac{\sqrt{2}|x'|}{\sqrt{1-|x_n|^2+\sqrt{|1-x_n|^2|1+x_n|^2-4x_{n,1}^2-4x_{n,2}^2}}},\]
where $x_n=x_{n,0}+ix_{n,1}+jx_{n,2}+kx_{n,3}$, with $x_{n,0}$, $x_{n,1}$, $x_{n,2}$, $x_{n,3}\in\mathbf{R}$.
The orbital metric on $\Delta$ is
\[ds^2=\frac{4}{1-u^2-v^2}\{(1-v^2)du^2+2uvdudv+(1-u^2)dv^2\},\]
that is, the (real) hyperbolic metric of constant curvature $-\frac{1}{4}$. Also, from Lemma~\ref{volume}, the volume functional at $(u,v)\in\Delta$ is
\[\mathcal{V}=\frac{(1+u^2)^3v^{4n-5}}{(1-u^2-v^2)^{\frac{4n+1}{2}}}.\]
Passing to polar coordinates $u=\tanh r\cos\theta$ and $v=\tanh r\sin\theta$, we can write
\begin{align*}
ds^2 &= 4(dr^2+\sinh^2rd\theta^2)\\
\mathcal{V} &= (\cosh^2r+\sinh^2r\cos^2\theta)(\sinh r)^{4n-5}(\sin\theta)^{4n-5},
\end{align*}
for $r\geq 0$ and $0\leq\theta\leq\pi$. Finally, Lemma~\ref{reduction-theorem} gives
\begin{proposition}\label{special-loxodromic-case}
Let $\gamma(s)=(r(s),\theta(s))$ be a curve in $\Delta$ parametrized by arc length, and let $\sigma(s)$ be the angle between $\frac{\partial}{\partial r}$ and the tangent unit vector $\frac{d\gamma}{ds}$. Then the hypersurface $\Gamma$ with $H\backslash\Gamma=\gamma$ has mean curvature $\mathbf{h}$ if and only if
\begin{equation}\label{special-loxodromic-ode}
\begin{aligned}
\frac{dr}{ds} &= \frac{1}{2}\cos\sigma\\
\frac{d\theta}{ds} &= \frac{1}{2}\frac{\sin\sigma}{\sinh r}\\
\frac{d\sigma}{ds} &= \frac{1}{2}\left((4n-5)\cot\theta-3\frac{\sinh^2r\sin 2\theta}{\cosh^2r+\sinh^2r\cos^2\theta}\right)\frac{\cos\sigma}{\sinh r}\\
&-\frac{1}{2}\left((4n-4)\coth r + 3\frac{\sinh 2r(1+\cos^2\theta)}{\cosh^2r+\sinh^2r\cos^2\theta}\right)\sin\sigma + \mathbf{h}.
\end{aligned}
\end{equation}
\end{proposition}
\begin{remark}\label{special-loxodromic-remark}
\begin{enumerate}[(i)]
\item\label{special-loxodromic-ode-boundary} System~(\ref{special-loxodromic-ode}) is singular on the boundary $\{(r,\theta):\mbox{$\theta=0$ or $\theta=\pi$}\}$. However, there exist solutions of (\ref{special-loxodromic-ode}) emanating orthogonally from points in there (see Lemma~\ref{boundary}).
\item\label{special-loxodromic-explicit-solutions} (\emph{Explicit solutions}) The orbital metric and the volume functional are invariant by the reflection on the line $\theta\equiv\frac{\pi}{2}$, hence this line is a solution of (\ref{special-loxodromic-ode}) with $\mathbf{h}\equiv0$. The generated hypersurface is the bisector in Example~\ref{canonical-bisector}.
\end{enumerate}
\end{remark}
\subsection{The parabolic case}\label{parabolic}
For each $m=1,\dots,n-1$ we identify the groups $\mathcal{H}^{4m-1}$ and $Sp(n-m)$ with the following subgroups of $Sp(n,1)$:
\[\left\{h(\xi,\nu)\in\mathcal{H}^{4n-1}:\mbox{$\xi=\begin{bmatrix}0\\\eta\end{bmatrix}$, with $\eta\in\mathbf{Q}^{m-1}$}\right\}\]
and
\[\left\{\begin{bmatrix}B&0\\0&I_{m+1}\end{bmatrix}:B\in Sp(n-m)\right\},\]
respectively. Let $H=\mathcal{H}^{4m-1}\times Sp(n-m)$. Using horospherical coordinates we see that
\[\Sigma=\{(\omega,\alpha,\beta):\mbox{$\beta=0$, $\omega_{n-m}\in\mathbf{R}$ and $\omega_l=0$ if $l\neq n-m$}\}\simeq\rhs^2\]
is a section for $(H,\qhs^n)$ and the orbit space is isometric to the quadrant
\[\Delta=\{(\alpha,\rho)\in\mathbf{R}^2:\mbox{$\alpha>0$ and $\rho\geq 0$}\},\]
where the canonical projection $\qhs^n\to\Delta$ is given by
\[(\omega,\alpha,\beta)\mapsto\left(\alpha,\sqrt{\sum_{l=1}^{n-m}|\omega_l|^2}\right).\]
The orbital metric on $\Delta$ is
\[ds^2=\frac{1}{\alpha^2}(d\alpha^2 + 4\alpha d\rho^2)\]
that is, the (real) hyperbolic metric of constant curvature $-\frac{1}{4}$. Also, from Lemma~\ref{volume}, the volume functional $(\alpha,\rho)\in\Delta$ is
\[\mathcal{V}(\alpha,\rho) = \alpha^{-\frac{4n+1}{2}}\rho^{4n-4m-1}.\]
Finally, Lemma~\ref{reduction-theorem} gives
\begin{proposition}\label{parabolic-case}
Let $\gamma(s)=(\alpha(s),\rho(s))$ be a curve in $\Delta$ parametrized by arc length, and let $\sigma(s)$ be the angle between $\frac{\partial}{\partial\alpha}$ and the tangent unit vector $\frac{d\gamma}{ds}$. Then the hypersurface $\Gamma$ with $H\backslash\Gamma=\gamma$ has mean curvature $\mathbf{h}$ if and only if
\begin{equation}\label{parabolic-ode}
\begin{aligned}
\frac{d\alpha}{ds} &= \alpha\cos\sigma\\
\frac{d\rho}{ds} &= \frac{1}{2}\sqrt{\alpha}\sin\sigma\\
\frac{d\sigma}{ds} &= \left(2n-2m-\frac{1}{2}\right)\frac{\sqrt{\alpha}}{\rho}\cos\sigma+(2n+1)\sin\sigma + \mathbf{h}.
\end{aligned}
\end{equation}
\end{proposition}
\begin{remark}\label{parabolic-remark}
\begin{enumerate}[(i)]
\item\label{parabolic-ode-boundary} System~(\ref{parabolic-ode}) is singular on the boundary $\{(\alpha,\rho):\mbox{$\alpha>0$ and $\rho=0$}\}$. However, there exist solutions of (\ref{parabolic-ode}) emanating orthogonally from points in there (see Lemma~\ref{boundary}).
\item\label{parabolic-explicit-solutions} (\emph{Explicit solutions}) Note that for $a>0$ the line $\alpha\equiv a$ is a solution of (\ref{parabolic-ode}) with $\mathbf{h}\equiv\pm(2n+1)$ which generates a horosphere. Also, since each $\psi_t\in\mathbf{R}_+$ normalizes $H$, it induces the transformation
\[(\alpha,\rho,\sigma)\mapsto(\mathrm{e}^{2t}\alpha,\mathrm{e}^t\rho,\sigma)\]
leaving (\ref{parabolic-ode}) invariant.
\end{enumerate}
\end{remark}
\subsection{The special parabolic case}\label{special-parabolic}
Let $H=\{h(\xi,\nu)\in\mathcal{H}^{4n-1}:\Re(\xi_{n-1})=0\}$. We have that $H$ acts freely on $\qhs^n$. Using horospherical coordinates we see that
\[\Sigma=\{(\omega,\alpha,\beta):\mbox{$\beta=0$, $\omega_{n-1}\in\mathbf{R}$ and $\omega_l=0$ if $l\neq n-1$}\}\simeq\rhs^2\]
is a section for $(H,\qhs^n)$ and the orbit space is isometric to the half-plane
\[\Delta=\{(\alpha,\rho)\in\mathbf{R}^2:\alpha>0\},\]
where the canonical projection $\qhs^n\to\Delta$ is given by
\[(\omega,\alpha,\beta)\mapsto(\alpha,\Re(\omega_{n-1})).\]
The orbital metric on $\Delta$ is
\[ds^2=\frac{1}{\alpha^2}(d\alpha^2 + 4\alpha d\rho^2)\]
that is, the (real) hyperbolic metric of constant curvature $-\frac{1}{4}$. Also, from Lemma~\ref{volume}, the volume functional $(\alpha,\rho)\in\Delta$ is
\[\mathcal{V}(\alpha,\rho) = \alpha^{-\frac{4n+1}{2}}.\]
Finally, Lemma~\ref{reduction-theorem} gives
\begin{proposition}\label{special-parabolic-case}
Let $\gamma(s)=(\alpha(s),\rho(s))$ be a curve in $\Delta$ parametrized by arc length, and let $\sigma(s)$ be the angle between $\frac{\partial}{\partial\alpha}$ and the tangent unit vector $\frac{d\gamma}{ds}$. Then the hypersurface $\Gamma$ with $H\backslash\Gamma=\gamma$ has mean curvature $\mathbf{h}$ if and only if
\begin{equation}\label{special-parabolic-ode}
\begin{aligned}
\frac{d\alpha}{ds} &= \alpha\cos\sigma\\
\frac{d\rho}{ds} &= \frac{1}{2}\sqrt{\alpha}\sin\sigma\\
\frac{d\sigma}{ds} &= (2n+1)\sin\sigma + \mathbf{h}.
\end{aligned}
\end{equation}
\end{proposition}
\begin{remark}\label{special-parabolic-remark}
\begin{enumerate}[(i)]
\item\label{special-parabolic-ode-boundary}
Note that all the orbits are principal, so the boundary of the orbit space is empty. Thus system (\ref{special-parabolic-ode}) has no singular points.  
\item\label{special-parabolic-explicit-solutions} (\emph{Explicit solutions}) Note that (\ref{special-parabolic-ode}) is invariant under $(\alpha,\rho,\sigma)\mapsto(\mathrm{e}^{2t}\alpha,\mathrm{e}^t\rho,\sigma)$, $t\in\mathbf{R}$ (induced by transvections as in Remark~\ref{parabolic-remark}~(\ref{parabolic-explicit-solutions})) and it is invariant under $\rho$-translations (induced by $\mathcal{H}^{4n-1}/H$). Also, for $R\in\mathbf{R}$, it is invariant under reflections on lines $\rho\equiv R$ (here $\mathbf{h}$ is taken to $-\mathbf{h}$). In particular, the lines $\rho\equiv R$ are solutions of \ref{special-parabolic-ode} with $\mathbf{h}\equiv0$ and they generate fans.
\end{enumerate}
\end{remark}
\section{Proof of Theorem~\ref{theorem-one}}
We write systems (\ref{elliptic-ode}) and (\ref{loxodromic-ode}) in an unified way and study their solutions for $\mathbf{h}\equiv0$. Of course, the volume functional is
\[\mathcal{V}=(\sinh r)^A(\sinh2r)^B(\sin\theta)^C(\sin2\theta)^D,\]
where $A$, $B$, $C$ and $D$ are positive integers depending on the specific transformation group. Thus for $\mathbf{h}\equiv0$ we have that (\ref{elliptic-ode}) and (\ref{loxodromic-ode}) are of the form
\begin{equation}\label{elliptic-loxodromic-ode}
\begin{aligned}
\frac{dr}{ds}=&\frac{1}{2}\cos\sigma\\
\frac{d\theta}{ds}=&\frac{1}{2}\frac{\sin\sigma}{\sinh r}\\
\frac{d\sigma}{ds}=&P(\theta)\frac{\cos\sigma}{\sinh r}-Q(r)\sin\sigma,
\end{aligned}
\end{equation}
where
\[P=\frac{1}{2}\frac{\partial}{\partial\theta}\ln\mathcal{V}=\frac{C}{2}\cot\theta+D\cot2\theta\]
and
\[Q=\frac{1}{2}\left(\frac{\partial}{\partial r}\ln\mathcal{V}+\coth r\right)=\frac{(A+1)}{2}\coth r +B\coth2r.\]
For each $a>0$ let $c_a(s)=(r_a(s),\theta_a(s),\sigma_a(s))$ be the solution  of (\ref{elliptic-loxodromic-ode}) with initial conditions $c_a(0)=(a,0,\frac{\pi}{2})$. We consider the one-parameter family of curves $\gamma_a(s)=(r_a(s),\theta_a(s))$, $a>0$. Next we will fix $a>0$ and we will study the global behavior of $\gamma_a$.

Multiplying the third equation in (\ref{elliptic-loxodromic-ode}) by $\sin2\theta$ and differentiating at $s=0$ we get that
\begin{equation}\label{e-l-sigma}
\frac{d\sigma_a}{ds}(0)=\frac{-Q(a)}{C+D+1}<0,
\end{equation}
since $Q>0$. On the other hand, from the third equation in (\ref{elliptic-loxodromic-ode}) it follows that
\[\frac{d\sigma}{ds}=
\begin{cases}
-Q(r)<0, &\text{ if }\sigma=\frac{\pi}{2}\\
Q(r)>0, &\text{ if }\sigma=-\frac{\pi}{2}
\end{cases}.\]
This, combined with (\ref{e-l-sigma}) implies that $\sigma_a(s)\in(-\frac{\pi}{2}+\delta,\frac{\pi}{2}-\delta)$ for some $\delta>0$. In particular the first equation in (\ref{elliptic-loxodromic-ode}) shows that $r_a$ is a strictly increasing function.

Now, consider the function
\[I=\mathcal{V}\cos\sigma.\]

Since $\frac{\partial}{\partial r}\mathcal{\ln V}=A\coth r+2B\coth2r$, we have $\frac{\partial}{\partial r}\mathcal{V}>(A+2B)\mathcal{V}=(4n+1)\mathcal{V}$. Thus, along a solution of (\ref{elliptic-loxodromic-ode}) 
\[\frac{dI}{ds}=\frac{1}{2}\left(\frac{\partial}{\partial r}\mathcal{V}+\mathcal{V}\coth r\sin^2\sigma\right)
>\frac{1}{2}\frac{\partial}{\partial r}\mathcal{V}>\frac{4n+1}{2}\mathcal{V}\geq\frac{4n+1}{2}I,\]
for $s>0$. Hence $\lim\limits_{s\to+\infty}I(s)=+\infty$. In particular, $\lim\limits_{s\to+\infty}r_a(s)=+\infty$ and $\theta_a(s)\in(0,\frac{\pi}{2})$, for all $s>0$.

We conclude that $c_a$ is a complete solution of (\ref{elliptic-loxodromic-ode}), hence $\gamma_a$ is defined for all $s\geq0$, and it does not have self-intersections. Therefore the generated hypersurface is a complete, embedded, minimal hypersurface in $\qhs^n$. In varying $a>0$, we get a one-parameter family of such hypersurfaces. We can also replace the chosen initial conditions by the initial conditions $c_a(0)=(a,\frac{\pi}{2},-\frac{\pi}{2})$ and repeat the argument in order to construct another one-parameter family of such hypersurfaces. This completes the proof of Theorem~\ref{theorem-one}, parts (\ref{theorem-one-i}) and (\ref{theorem-one-ii}).
\section{Proof of Theorem~\ref{theorem-two}}
We analyse the global behavior of solution curves of (\ref{special-loxodromic-ode}) for $\mathbf{h}\equiv0$. For all $a\in\mathbf{R}$ let $c_a(s)=(r_a(s),\theta_a(s),\sigma_a(s))$ be the solution with initial conditions
\[c_a(0)=
\begin{cases}
(a,0,\frac{\pi}{2}), &\text{ if }a>0\\
(0,\frac{\pi}{2},0), &\text{ if }a=0\\
(-a,\pi,-\frac{\pi}{2}), &\text{ if }a<0
\end{cases}.\]
Set $\gamma_a(s)=(r_a(s),\theta_a(s))$. We have that $\gamma_0$ is the \emph{bisector solution} $\theta\equiv\frac{\pi}{2}$ and $\gamma_{-a}$ is the mirror image of $\gamma_a$ on $\theta\equiv\frac{\pi}{2}$ (see Remark~\ref{special-loxodromic-remark}~(\ref{special-loxodromic-explicit-solutions})). Then it is sufficient consider $a>0$. Proceeding as in the proof of Theorem~\ref{theorem-one} and following the same steps, using even the same semi-first integral $I$, we show that each $\gamma_a(s)$ is defined for all $s\geq0$ without self-intersections.

In order to show that the family $(c_a)_{a\in\mathbf{R}}$ is analytic in $a$, we go back to Cartesian coordinates $(u,v)$ as in $\S$\ref{special-loxodromic} and note that near to line $v=0$ we may consider $u$ as a function of $v$, so that (\ref{special-loxodromic-ode}) becomes:
\begin{equation}\label{ode-analytic}
\begin{aligned}
v\frac{d^2u}{dv^2} = \frac{1}{(1+u^2)(1-u^2-v^2)}
&\left\{(1-v^2)\left(\frac{du}{dv}\right)^2+2uv\frac{du}{dv}+1-u^2\right\}\\\times
&\left\{12uv-((4n-5)(1+u^2)+6v^2)\frac{du}{dv}\right\}.
\end{aligned}
\end{equation}
Lemma~\ref{analytic} shows that there exists a unique analytic solution $u=u(t,v)$ of (\ref{ode-analytic}) which is a convergent power series of $(t,v)$ in a neighborhood of $(t_0,0)$ with $u(t_0,0)=t_0$, $\frac{du}{dv}(t_0,0)=0$, for any $t_0\in(-1,1)$.

The hypersurface generated by $\gamma_0$ is given by the equation $u=\Re(kx_n)=0$, hence it is the bisector $\mathcal{B}$ in Example~\ref{canonical-bisector}. Moreover, no other curve $\gamma_a$ generates a bisector. In fact, suppose that some $\gamma_a$ generate the bisector $\mathcal{B}'$. The group $H=N\cdot\mathbf{R}_+\cdot Sp(n-1)$ stabilizes $\mathcal{B}'$, so $H$ stabilizes its spine $\sigma'$. In particular $\psi_t(\sigma')=\sigma'$, for all $t\in\mathbf{R}$. So (passing to horospherical coordinates) if $p\in\sigma'$ we get that
\[\mbox{$\lim\limits_{t\to-\infty}\psi_t(p)=(0,0,0)$ and $\lim\limits_{t\to+\infty}\psi_t(p)=\infty$}.\]
Then $(0,0,0)$, $\infty\in\partial\sigma'\subset\partial\Sigma'$. Thus we see that $\mathcal{B}$ and $\mathcal{B}'$ have the same $\mathbf{Q}$-spine $\Sigma$, since two distinct points in $\partial\qhs^n$ determine a unique $\mathbf{Q}$-line. Moreover, as $\sigma'$ is totally geodesic then the unique real geodesic with endpoints $(0,0,0)$ and $\infty$ is contained in $\sigma'$. In particular  the base-point $(0,1,0)$ is in $\sigma'$. On the other hand, $N\cdot\mathbf{R}_+$ acts on $\Sigma$ by isometries stabilizing $\sigma'$ and the $N\cdot\mathbf{R}_+$\nobreakdash -orbit through $(0,1,0)$ is the spine of $\mathcal{B}$. Hence the spines of $\mathcal{B}$ and $\mathcal{B}'$ coincide. Therefore $\mathcal{B}=\mathcal{B}'$, i.e. $a=0$.

We have shown that the hypersurfaces generated by $\gamma_a$, $a\in\mathbf{R}$, define an analytic one-parameter family of embedded, complete, minimal hypersurfaces diffeomorphic to $\mathbf{R}^{4n-1}$, such that the hypersurface corresponding to $a=0$ is the unique bisector in the family. Since the isometry group of $\qhs^n$ acts transitively on the set of bisectors, this completes the proof of Theorem~\ref{theorem-two}.
\section{Proof of Theorem~\ref{theorem-three}}
\subsection{The parabolic case}
We want to analyse the global behavior of solutions of system~(\ref{parabolic-ode}) for $\mathbf{h}\equiv0$. Consider the solution $c_a(s)=(\alpha_a(s),\rho_a(s),\sigma_a(s))$ with initial conditions $c_a(0)=(a,0,\frac{\pi}{2})$, for all $a>0$. We also write $\gamma_a(s)=(\alpha_a(s),\rho_a(s))$. As before, we fix $a>0$ and study the curve $\gamma_a$.

Multiplying the third equation in (\ref{parabolic-ode}) by $\rho$ and differentiating at $s=0$ we have    that
\begin{equation}\label{p-sigma-d1}
\frac{d\sigma}{ds}(0)=\frac{2n+1}{4n-4m}>0.
\end{equation}
Next, differentiating third equation in (\ref{parabolic-ode}), we get that
\begin{equation}\label{p-sigma-d2}
\mbox{$\frac{d^2\sigma}{ds^2}=\frac{\sqrt{\alpha}}{2\rho}\left\{\left(2n-2m-\frac{1}{2}\right)\cos^2\sigma+(2n+1)\sin^2\sigma\right\}>0$, if $\frac{d\sigma}{ds}=0$}.
\end{equation}
On the other hand, from the third equation in (\ref{parabolic-ode}) we have that
\begin{equation}\label{p-sigma-d3}
\frac{d\sigma}{ds}=-\left(2n-2m-\frac{1}{2}\right)\frac{\sqrt{\alpha}}{\rho}<0
\end{equation}
for $\sigma=\pi$. Assertions (\ref{p-sigma-d1}), (\ref{p-sigma-d2}) and (\ref{p-sigma-d3}) combined imply that $\sigma_a$ is monotonically increasing and $\lim\limits_{s\to+\infty}\sigma_a(s)\in(\frac{\pi}{2},\pi]$.

In particular, $\frac{d\alpha_a}{ds}(s)<0$ and $\frac{d\rho_a}{ds}(s)>0$. The first inequality says that $\lim\limits_{s\to+\infty}\alpha_a(s)\geq0$. But the first equation in (\ref{parabolic-ode}) implies that $\lim\limits_{s\to+\infty}\alpha_a(s)=0$ since $\lim\limits_{s\to+\infty}\frac{d\alpha_a}{ds}(s)=0$. The second inequality says that $\left|\frac{\cos\sigma_a}{\rho_a}\right|$ is bounded. Since $\lim\limits_{s\to+\infty}\alpha_a(s)=\lim\limits_{s\to+\infty}\frac{d\sigma_a}{ds}(s)=0$, from the third equation in (\ref{parabolic-ode}) we must have that $\lim\limits_{s\to+\infty}\sigma_a(s)=\pi$.

So far we have shown that $\alpha_a$ monotonically decreases to $0$, $\rho_a$ is monotonically increasing, and $\sigma_a$ monotonically increases to $\pi$. Next we show that $\rho_a$ is bounded and estimate $\lim\limits_{s\to+\infty}\rho_a(s)$ to see that $(\gamma_a)_{a>0}$ fills the orbit space.

In fact, let
\[I=\alpha^{-2n-1}\rho^{\frac{(4n+2)(4n-4m-2)+1}{4n+1}}
\left\{\sqrt{\alpha}\cos\sigma+\frac{4n+1}{4n-4m-1}\rho\sin\sigma\right\}\]
and
\[J=\alpha^{\frac{-(4n-4m-1)(4n+3)-1}{8n-8m}}\rho^{4n-4m-1}
\left\{\sqrt{\alpha}\cos\sigma+\frac{4n+2}{4n-4m}\rho\sin\sigma\right\}.\]
Then along a solution of (\ref{parabolic-ode}):
\[\frac{dI}{ds}=-\left(\frac{2m+2}{4n+1}\right)\alpha^{-2n}\rho^{\frac{(4n+2)(4n-4m-2)-4n}{4n+1}}\sin\sigma\cos\sigma>0\]
and
\[\frac{dJ}{ds}=\frac{(4m+2)(4m+1)}{(4n-4m)^2}
\alpha^{\frac{-(4n-4m-1)(4n+3)-1}{8n-8m}}\rho^{4n-4m}\sin\sigma\cos\sigma<0,\]
for all $s>0$. In particular, as $I(0)=J(0)=0$, we have $J(s)<0<I(s)$ for all $s>0$. Then
\[\frac{-4n+4m}{4n+2}\frac{\sqrt{\alpha_a}}{\rho_a}<\tan\sigma_a<
\frac{-4n+4m+1}{4n+1}\frac{\sqrt{\alpha_a}}{\rho_a}.\]
This combined with the first two equations in (\ref{parabolic-ode}) gives
\[\frac{-4n+4m+1}{4n+1}\frac{d\alpha_a}{ds}<2\rho_a\frac{d\rho_a}{ds}<\frac{-4n+4m}{4n+2}\frac{d\alpha_a}{ds}.\]
Integrating the above inequalities on $[0,s]$ and making $s\to+\infty$ we obtain
\begin{equation*}
\sqrt{a\frac{4n-4m-1}{4n+1}}\leq\lim\limits_{s\to\infty}\rho_a(s)\leq\sqrt{a\frac{4n-4m}{4n+2}}.
\end{equation*}
Next we show that the family $(\gamma_a)_{a>0}$ forms a foliation of the orbit space. For this we consider the foliation of the orbit space for arcs of the parabolae $\alpha=q^2\rho^2$, $q\in(0,+\infty]$. We already know that each $\gamma_a$ must cut across all these arcs exactly once. On the other hand, Remark~\ref{parabolic-remark}~(\ref{parabolic-explicit-solutions}) says that (\ref{parabolic-ode}) is invariant by dilatations $(\alpha,\rho)\mapsto (r^2\alpha,r\rho)$ fixing $\sigma$. Since this one-parameter group leave each arc of parabola invariant is clear that $\gamma_a$ and $\gamma_{a'}$ cannot mutually intersect if $a\not=a'$.

Finally, let $\Gamma_a$ be the hypersurface of $\qhs^n$ generated by $\gamma_a$. The arguments above show that the family $(\Gamma_a)_{a>0}$ form a transvection-invariant, minimal foliation with leaf diffeomorphic to $\mathbf{R}^{4n-1}$. The ideal boundary $\partial\Gamma_a$ is the closure in $\partial\qhs^n$ of the $H$-orbit of $\lim\limits_{s\to+\infty}\gamma_a(s)=(0,R_a)$. In the Siegel domain we have that
\begin{align*}
\partial\Gamma_a =& \left\{\zeta\in\partial\mathcal{S}^n:\sum_{l=1}^{n-m}|\zeta_l|^2=R_a^2\right\}\cup\{\infty\}\\
=& \left\{\zeta\in\mathbf{Q}^n:
\sum_{l=1}^{n-m}|\zeta_l|^2=R_a^2=2\Re(\zeta_n)-\sum_{l=n-m+1}^{n-1}|\zeta_l|^2\right\}\cup\{\infty\}.
\end{align*}
Therefore $\partial\Gamma_a$ is a pinched Hopf manifold of type $(4m-1,4n-4m-1)$. This completes the proof of Theorem~\ref{theorem-three} part (\ref{theorem-three-i}).
\subsection{The special parabolic case}
We next analyse system~(\ref{special-parabolic-ode}) for $\mathbf{h}\equiv0$. The function $I=\alpha^{-2n-1}\sin\sigma$ is a first integral of (\ref{special-parabolic-ode}), i.e. it is constant along any solution curve. Let $c(s)=(\alpha(s),\rho(s),\sigma(s))$ the solution with initial conditions $c(0)=(1,0,\frac{\pi}{2})$. Then $I\equiv1$ along $c$, so $\sigma(s)\in(0,\pi)$ for all $s>0$. From the third equation in (\ref{special-parabolic-ode}) we get that $\frac{d\sigma}{ds}>0$, so $\lim\limits_{s\to+\infty}\sigma(s)\in(\frac{\pi}{2},\pi]$. Thus, it follows from the first equation in (\ref{special-parabolic-ode}) that $\lim\limits_{s\to+\infty}\alpha(s)=0$. But $I(s)=1$, then $\lim\limits_{s\to+\infty}\sigma(s)=\pi$. Finally, $I\equiv1$ gives that $\sin\sigma=\alpha^{2n+1}$ and substituting into $\frac{d\rho}{d\alpha}=\frac{\tan\alpha}{2\sqrt{\alpha}}$ yields
\[\rho(\alpha)=\frac{1}{2}\int_{\alpha}^1\sqrt{\frac{t^{4n+1}}{1-t^{4n+2}}}\,dt,\]
for $0<\alpha\leq1$. This is an elliptic integral, convergent at $t=1$. The graph of $\rho=\rho(\alpha)$ can be continued to a complete solution curve of (\ref{special-parabolic-ode}) by reflection on the line $\rho=0$. This gives a solution curve generating a minimal hypersurface $\Gamma$ in $\qhs^n$ diffeomorphic to $\mathbf{R}^{4n-1}$, whose ideal boundary $\partial\Gamma$ is the closure in $\partial\qhs^n$ of the $H$-orbit of the pair of points $(0,\pm R)$, where $R=\rho(0)$. In the Siegel domain we have that
\[\partial\Gamma=\{\zeta\in\partial\mathcal{S}:\Re(\zeta_{n-1})=R\}\cup\{\infty\}\cup
\{\zeta\in\partial\mathcal{S}:\Re(\zeta_{n-1})=-R\}.\]
So $\partial\Gamma$ is a bouquet of two spheres, glued at the point at infinity. Finally, by applying transvections (see Remark~\ref{special-parabolic-remark}~(\ref{special-parabolic-explicit-solutions})) to $\Gamma$ we get the desired foliation of $\qhs^n$, and this completes the proof of Theorem~\ref{theorem-three}, part (\ref{theorem-three-ii}).

From Remark~\ref{special-parabolic-remark}~(\ref{special-parabolic-explicit-solutions}), the lines $\rho \equiv R$ define a $\rho$-translation invariant foliation of the orbit space by solution curves of (\ref{special-parabolic-ode}) and the leaves of the corresponding foliation on $\qhs^n$ are fans. This, together with Proposition~\ref{fans}, completes the proof of Theorem~\ref{theorem-three}, part (\ref{theorem-three-iii}).
\begin{acknowledgments}
The author would like to thank Prof.~Claudio Gorodski for useful discussions and several valuable comments.
\end{acknowledgments}
\bibliography{cmhqhs-ams-5}

\providecommand{\bysame}{\leavevmode\hbox to3em{\hrulefill}\thinspace}
\providecommand{\MR}{\relax\ifhmode\unskip\space\fi MR }
\providecommand{\MRhref}[2]{%
  \href{http://www.ams.org/mathscinet-getitem?mr=#1}{#2}
}
\providecommand{\href}[2]{#2}
\begin{thebibliography}{BdCH09}

\bibitem[AK07]{apanasov-kim1}
B.~Apanasov and I.~Kim, \emph{Cartan angular invariant and deformations in rank
  one symmetric spaces}, Sb. Math. \textbf{198} (2007), no.~2-3, 147--169.

\bibitem[BdCH09]{back-docarmo-hsiang1}
A.~Back, M.~do~Carmo, and W.-Y. Hsiang, \emph{On some fundamental equations of
  equivariant {R}iemannian geometry}, Tamkang J. Math. \textbf{40} (2009),
  no.~4, 343--376.

\bibitem[Ber98]{berndt1}
J.~Berndt, \emph{Homogeneous hypersurfaces of hyperbolic spaces}, Math. Z.
  \textbf{229} (1998), 589--600.

\bibitem[Che73]{chen1}
S-S. Chen, \emph{On subgroups of the noncompact real exceptional {L}ie group
  ${F}_4^*$}, Math. Ann. \textbf{204} (1973), 271--284.

\bibitem[dCD83]{docarmo-dajczer1}
M.~do~Carmo and M.~Dajczer, \emph{Rotation hypersurfaces in spaces of constant
  curvature}, Trans. Amer. Math. Soc. \textbf{277} (1983), 685--709.

\bibitem[GG00]{gorodski-gusevskii1}
C.~Gorodski and N.~Gusevskii, \emph{Complete minimal hypersurfaces in complex
  hyperbolic space}, Manuscripta Math. \textbf{103} (2000), 221--240.

\bibitem[Gir21]{giraud}
G.~Giraud, \emph{Sur certaines fonctions automorphes de deux variables}, Ann.
  Sci. \'Ec. Norm. Sup\'er. \textbf{38} (1921), no.~3, 43--164.

\bibitem[Gol06]{goldman0}
W.M. Goldman, \emph{Complex hyperbolic geometry}, Oxford Mathematical
  Monographs, Oxford University Press, Oxford, New York, 2006.

\bibitem[GP92a]{goldman-parker0}
W.M. Goldman and J.R. Parker, \emph{Complex hyperbolic space ideal triangle
  groups}, J. Reine Angew. Math. \textbf{425} (1992), 71--86.

\bibitem[GP92b]{goldman-parker1}
\bysame, \emph{Dirichlet polyhedra for dihedral groups acting on complex
  hyperbolic space}, J. Geom. Anal. \textbf{2} (1992), no.~6, 517--553.

\bibitem[GP00]{gusevskii-parker0}
N.~Gusevskii and J.R. Parker, \emph{Representations of free fuchsian groups in
  complex hyperbolic space}, Topology \textbf{39} (2000), 33--60.

\bibitem[GP03]{gusevskii-parker1}
N.~Gusevskii and J.~R. Parker, \emph{Complex hyperbolic quasi-fuchsian groups
  and toledo's invariant}, Geom. Dedicata \textbf{97} (2003).

\bibitem[HH82]{hsiang-hsiang2}
W.-T. Hsiang and W.-Y. Hsiang, \emph{On the existence of codimension-one
  minimal spheres in compact symmetric spaces of rank 2. {II}}, J. Differential
  Geom. \textbf{17} (1982), 583--594.

\bibitem[HL71]{hsiang-lawson1}
W.-Y. Hsiang and H.B.~Jr Lawson, \emph{Minimal submanifolds of low
  cohomogeneity}, J. Differential Geom. \textbf{5} (1971), 1--38.

\bibitem[Hsi85]{hsiang}
W.-Y. Hsiang, \emph{On the construction of constant mean curvature embeddings
  of exotic and{/}or knotted spheres into ${S}^n(1)$}, Invent. Math.
  \textbf{82} (1985), 423--445.

\bibitem[Kol11]{kollross}
A.~Kollross, \emph{Duality of symmetric spaces and polar actions}, J. Lie
  Theory \textbf{21} (2011), no.~4.

\bibitem[Mos80]{mostow1}
G.D. Mostow, \emph{On a remarkable class of polyhedra in complex hyperbolic
  space}, Pacific. J. Math. \textbf{86} (1980), 171--276.

\bibitem[PT88]{palais-terng}
R.S. Palais and Ch.-L. Terng, \emph{Critical point theory and submanifold
  theory}, Lecture Notes in Mathematics, vol. 1353, ch.~5, Springer-Verlag,
  Berlin-Heidelberg-New York-London-Paris-Tokyo, 1988.

\end{thebibliography}
\bibliographystyle{amsalpha}
\end{document}